\documentclass[psreqno,reqno,11pt]{amsart}
\usepackage{eurosym}
\usepackage{amssymb}
\usepackage{tikz}
\usepackage{amsmath}
\usepackage{tikz}
\usepackage{pgflibraryplotmarks}
\usepackage{wasysym}
\usepackage{stmaryrd}
\usepackage[english]{babel}

\usepackage{hyperref}

\theoremstyle{plain}
\newtheorem{theorem}{Theorem}[section]

\newtheorem{lemma}[theorem]{Lemma}

\theoremstyle{definition}

\begin{document}
\def\N{\mathbb{N}}
\def\Z{\mathbb{Z}}
\def\R{\mathbb{R}}
\def\C{\mathbb{C}}
\def\w{\omega}

\title[Some Notes on Orthogonally Additive Polynomials]{
Some Notes on Orthogonally Additive Polynomials}
\author{C. Schwanke}
\address{Department of Mathematics, Lyon College, Batesville, AR, USA and Unit for BMI, North-West University, Private Bag X6001, Potchefstroom, 2520, South Africa}
\email{cmschwanke26@gmail.com}
\date{\today}

\begin{abstract}
We provide two new characterizations of bounded orthogonally additive polynomials from a uniformly complete vector lattice into a convex bornological space using separately two polynomial identities of Kusraeva involving the root mean power and the geometric mean. Furthermore, it is shown that a polynomial on a vector lattice is orthogonally additive whenever it is orthogonally additive on the positive cone. These results improve recent characterizations of bounded orthogonally additive polynomials by G. Buskes and the author.
\end{abstract}

\maketitle
\section{Introduction}\label{S:intro}

The $n$th root mean power $\mathfrak{S}_n$ and the $n$th geometric mean $\mathfrak{G}_n$ are defined as
\[
\mathfrak{S}_n(x_1,\dots,x_r)=\sqrt[n]{\sum_{k=1}^{r}x_k^n}\quad (x_1,\dots,x_r\in\mathbb{R})
\]
and
\[
\mathfrak{G}_{n}(x_1,\dots,x_n)=\sqrt[n]{\prod_{k=1}^{n}|x_k|}\quad (x_1,\dots,x_n\in\mathbb{R}),
\]
respectively. In \cite{Kusa}, Kusraeva uses the Archimedean vector lattice functional calculus, as developed in \cite{BusdPvR}, to define $\mathfrak{S}_{n}$ and $\mathfrak{G}_{n}$ in uniformly complete vector lattices. It is proven in \cite{Kusa} that if (i) $E$ is a uniformly complete vector lattice, (ii) $Y$ a convex bornological space, and (iii) $P\colon E\to Y$ is a bounded orthogonally additive $n$-homogeneous polynomial with unique corresponding symmetric $n$-linear map $\check{P}$, then the following hold:
\begin{equation}\label{eq:RMP}
P(\mathfrak{S}_{n}(f_{1},\dots,f_{r}))=\sum_{k=1}^rP(f_{k})\quad (f_{1},\dots,f_{r}\in E^+, r\in\mathbb{N}\setminus\{1\})
\end{equation}
and
\begin{equation}\label{eq:GM}
P(\mathfrak{G}_{n}(f_{1},\dots,f_{n}))=\check{P}(f_{1},\dots,f_{n})\quad (f_{1},\dots,f_{n}\in E^+).
\end{equation}

The story continues in \cite[Theorems 2.3\&2.4]{BusSch4}, where it is shown that the attainment of \textit{both} \eqref{eq:RMP} and \eqref{eq:GM} above provides a single characterization of bounded orthogonally additive polynomials $P\colon E\to Y$. The purpose of this paper is to illustrate that attainment of either \eqref{eq:RMP} or \eqref{eq:GM} \textit{alone} characterizes bounded orthogonally additive polynomials $P\colon E\to Y$. This in turn proves that \eqref{eq:RMP} and \eqref{eq:GM} are actually equivalent in this setting and considerably improves the aforementioned \cite[Theorems 2.3\&2.4]{BusSch4}.

The novel approach in this paper is that we express $\mathfrak{S}_n$ and $\mathfrak{G}_n$ (for $n>1$) in terms of convenient explicit formulas rather than relying solely on functional calculus. This is possible using \cite[Theorem~3.7]{BusSch}, as $\mathfrak{S}_n$ is convex on the positive cone, while $\mathfrak{G}_n$ is concave on the positive cone. Indeed, for a uniformly complete vector lattice $E$ and $n,r\in\mathbb{N}\setminus\{1\}$, we show in the proof of Theorem~\ref{T:RMP+GM} that
\[
\mathfrak{S}_n(f_1,\dots,f_r)=\sup\left\{\sum_{k=1}^ra_kf_k\ :\ a_1,\dots,a_r\in[0,1],\ \sum_{k=1}^ra_k^m=1\right\}
\]
holds for all $f_1,\dots, f_r\in E^+$, where $m$ is the H\"older conjugate of $n$, and $\mathfrak{S}_n(f_1,\dots,f_r)$ is defined via functional calculus. Furthermore, it was proven in \cite[Corollary~3.9]{BusSch} that
\[
\mathfrak{G}_n(f_1,\dots, f_n)=\dfrac{1}{n}\inf\left\{\sum_{k=1}^{n}\theta_kf_k\ :\ \theta_1,...,\theta_n\in(0,\infty),\ \prod_{i=1}^{n}\theta_i=1\right\}
\]
holds for all $f_1,\dots, f_n\in E^+$, where again functional calculus is used to define $\mathfrak{G}_n(f_1,\dots, f_n)$. These explicit formulas greatly aid the obtainment of our results presented in this paper.

We as usual denote the set of strictly positive integers by $\mathbb{N}$ and the ordered field of real numbers by $\mathbb{R}$. All vector spaces in this manuscript are real, and all vector lattices are Archimedean. For any unexplained terminology, notation, or basic theory regarding vector lattices, we refer the reader to the standard texts \cite{AB,LuxZan1,Zan2}.

Let $E$ be a uniformly complete vector lattice, let $V$ be a vector space, and put $n\in\mathbb{N}$. Recall that a map $P\colon E\to V$ is called an $n$-\textit{homogeneous polynomial} if there exists a (unique) symmetric $n$-linear map $\check{P}\colon E^n\to V$ such that $P(f)=\check{P}(f,\dots,f)\ (f\in E)$. (We denote the symmetric $n$-linear map associated with an $n$-homogeneous polynomial $P$ by $\check{P}$ throughout.) Given an $n$-homogeneous polynomial $P\colon E\to V$, $r\in\mathbb{N}$ with $r\leq n$, $f_1,\dots,f_r\in E$, and $k_1,\dots,k_r\in\{0,\dots,n\}$ satisfying $\sum_{i=1}^{r}k_i=n$ we will write
\[
\check{P}(f_1^{k_1}f_2^{k_2}\cdots f_r^{k_r}):=\check{P}(\underbrace{f_1,\dots,f_1}_{k_1\ \text{copies}},\underbrace{f_2,\dots,f_2}_{k_2\ \text{copies}},\dots,\underbrace{f_r,\dots,f_r}_{k_r\ \text{copies}})
\]
and will use similar notation with $\mathfrak{S}_n$ and $\mathfrak{G}_n$ as well. Finally, recall that an $n$-homogeneous polynomial $P\colon E\to V$ is said to be \textit{orthogonally additive} if
\[
P(f+g)=P(f)+P(g)
\]
holds whenever $f,g\in E$ are disjoint. We will also say that $P$ is \textit{positively orthogonally additive} if $P(f+g)=P(f)+P(g)$ holds whenever $f,g\in E^+$ are disjoint.

\section{Main Results}\label{S:MR}

The following lemma is needed in order to obtain the main results in this section.

\begin{lemma}\label{L:posoaeqop}
Let $n\in\mathbb{N}\setminus\{1\}$, let $E$ be a vector lattice, and suppose that $V$ is a vector space. Assume $P\colon E\to V$ is an $n$-homogeneous polynomial. Then $P$ is orthogonally additive if and only if $P$ is positively orthogonally additive.
\end{lemma}

\begin{proof}
We need to prove only the nontrivial implication. For this task, assume $P$ is positively orthogonally additive. From the binomial theorem we have
\[
P(f+\lambda g)=P(f)+P(\lambda g)+\sum_{k=1}^{n-1}\binom{n}{k}\lambda^k\check{P}(f^{n-k}g^k)\quad (f,g\in E, \lambda\in\mathbb{R}).
\]
Since $P$ is positively orthogonally additive, it follows that
\[
\sum_{k=1}^{n-1}\binom{n}{k}\lambda^k\check{P}(f^{n-k}g^k)=0
\]
for all $f,g\in E^+$ disjoint and all $\lambda\in\mathbb{R}^+$. From \cite[Lemma~2.1]{BusSch4} we obtain
\begin{equation}\label{eq:termsarezero}
\check{P}(f^{n-k}g^k)=0\quad (f,g\in E^+\ \text{with}\ f\perp g,\ k\in\{1,\dots,n-1\}).
\end{equation}
Using \eqref{eq:termsarezero}, we show that
\[
P(f)=P(f^+)+P(-f^-)
\]
holds for every $f\in E$. To this end, let $f\in E$. Using again the binomial theorem as well as \eqref{eq:termsarezero} above, we get
\begin{align*}
P(f)&=P(f^+-f^-)\\
&=\sum_{k=0}^{n}\binom{n}{k}\check{P}\Bigl((f^+)^{n-k}(-f^-)^k\Bigr)\\
&=\sum_{k=0}^{n}\binom{n}{k}(-1)^k\check{P}\Bigl((f^+)^{n-k}(f^-)^k\Bigr)\\
&=P(f^+)+(-1)^nP(f^-).
\end{align*}
Furthermore, if $n$ is even, then $P$ is even, and so
\[
P(f^+)+(-1)^nP(f^-)=P(f^+)+P(f^-)=P(f^+)+P(-f^-).
\]
On the other hand, if $n$ is odd, then $P$ is odd, and thus
\[
P(f^+)+(-1)^nP(f^-)=P(f^+)-P(f^-)=P(f^+)+P(-f^-).
\]
Thus $P(f)=P(f^+)+P(-f^-)$ holds for all $f\in E$, as claimed. Exploiting this fact, we prove that $P$ is orthogonally additive. To this end, let $f,g\in E$ be disjoint. If $n$ is even, we have
\begin{align*}
P(f+g)&=P\Bigl((f+g)^+\Bigr)+P\Bigl(-(f+g)^-\Bigr)\\
&=P\Bigl((f+g)^+\Bigr)+P\Bigl((f+g)^-\Bigr)\\
&=P(f^++g^+)+P(f^-+g^-)\\
&=P(f^+)+P(f^-)+P(g^+)+P(g^-)\\
&=P(f^+)+P(-f^-)+P(g^+)+P(-g^-)\\
&=P(f)+P(g).
\end{align*}
A similar argument handles the case that $n$ is odd. Therefore, $P$ is orthogonally additive.
\end{proof}

We proceed to our main results.

\begin{theorem}\label{T:RMP+GM}
Let $n,r\in\mathbb{N}\setminus\{1\}$. Suppose $E$ is a uniformly complete vector lattice, $V$ is a vector space, and $P\colon E\to V$ is an $n$-homogeneous polynomial.
\begin{itemize}
\item[(i)] If $P(\mathfrak{S}_{n}(f_{1},\dots,f_{r}))=\sum_{k=1}^rP(f_{k})$ holds for every $f_{1},\dots,f_{r}\in E^{+}$, then $P$ is orthogonally additive.
\item[(ii)] If $P(\mathfrak{G}_{n}(f_{1},\dots,f_{n}))=\check{P}(f_{1},\dots,f_{n})$, holds for all $f_{1},\dots,f_{n}\in E^+$, then $P$ is orthogonally additive.
\end{itemize}
\end{theorem}

\begin{proof}
We claim that
\[
\mathfrak{S}_n(f_1,\dots,f_r)=\sup\left\{\sum_{k=1}^ra_kf_k\ :\ 0\leq a_1,\dots,a_r\leq 1,\ \sum_{k=1}^ra_k^m=1\right\}
\]
holds for all $f_1,\dots, f_r\in E^+$, where $m$ is the H\"older conjugate of $n$; i.e. $m^{-1}+n^{-1}=1$. To this end, let $f_1,\dots, f_r\in E^+$. Below $\mathfrak{S}_n(f_1,\dots, f_r)$ is defined via the Archimedean vector lattice functional calculus. However, for $c_1,\dots,c_r\in\mathbb{R}$, the $n$th root mean power is defined classically:
\[
\mathfrak{S}_n(c_1,\dots,c_r)=\sqrt[n]{\sum_{k=1}^{r}c_k^n}.
\]
(This type of notation is standard when using the Archimedean vector lattice functional calculus.) It follows from \cite[Theorem~3.7(1)]{BusSch} that
\begin{align*}
	\mathfrak{S}_n(f_1,\dots,f_r)&=\sup\left\{\sum_{k=1}^{r}\dfrac{\partial\mathfrak{S}_n}{\partial x_k}(c_1,\dots,c_r)f_k\ :\ c_1,\dots,c_r\in\mathbb{R}^+,\ \sum_{k=1}^{r}c_k^2=1\right\}\\
	&=\sup\left\{\sum_{k=1}^{r}\dfrac{c_k^{n-1}}{\left(\sum_{i=1}^rc_i^n\right)^{1/m}}f_k\ :\ c_1,\dots,c_r\in\mathbb{R}^+,\ \sum_{k=1}^{r}c_k^2=1\right\}\\
	&=\sup\left\{\sum_{k=1}^{r}a_kf_k\ :\ a_1,\dots,a_r\in[0,1],\ \sum_{k=1}^ra_k^m=1\right\}.
\end{align*}
Next let $f,g\in E^+$ with $f$ and $g$ disjoint. We illustrate that 
\[
f+g=\mathfrak{S}_n(fg0^{r-2}).
\]
Indeed, using that $f$ and $g$ are disjoint and positive, we have
\begin{align*}
	f+g&=f\vee g\\
	&\leq\sup\left\{a_1f+a_2g\ :\ a_1,a_2\in[0,1],\ a_1^m+a_2^m=1\right\}\\
	&\leq f+g.
\end{align*}
However,
\[
\mathfrak{S}_n(fg0^{r-2})=\sup\left\{a_1f+a_2g\ :\ a_1,a_2\in[0,1],\ a_1^m+a_2^m=1\right\}
\]
holds from the explicit formula for $\mathfrak{S}_n$ given above. Thus $f+g=\mathfrak{S}_n(fg0^{r-2})$, as claimed. Next let $P\colon E\to V$ be an $n$-homogeneous polynomial satisfying
\[
P(\mathfrak{S}_{n}(f_{1},\dots,f_{r}))=\sum_{k=1}^rP(f_{k})\quad (f_{1},\dots,f_{r}\in E^{+}).
\]
It follows from our argument above that
\begin{align*}
P(f+g)&=P(\mathfrak{S}_n(fg0^{r-2}))=P(f)+P(g).
\end{align*}
Hence $P$ is positively orthogonally additive. By Lemma~\ref{L:posoaeqop}, we have that $P$ is orthogonally additive. This completes the proof of (i).

To prove (ii), let $f,g\in E^+$ be disjoint, and put $k\in\{1,\dots,n-1\}$. By \cite[Corollary~3.9]{BusSch}, we have
\[
\mathfrak{G}_n(f^{n-k}g^k)=\dfrac{1}{n}\inf\left\{\sum_{i=1}^{n-k}\theta_if+\sum_{i=n-k+1}^{n}\theta_ig\ :\ \theta_1,\dots,\theta_n\in(0,\infty),\ \prod_{i=1}^{n}\theta_i=1\right\}.
\]
Clearly, $\mathfrak{G}_n(f^{n-k}g^k)\geq 0$. Suppose that
\[
l\leq\sum_{i=1}^{n-k}\theta_if+\sum_{i=n-k+1}^{n}\theta_ig
\]
holds for all $\theta_1,\dots,\theta_n\in(0,\infty)$ for which $\prod_{i=1}^{n}\theta_i=1$. Then $l=l_1+l_2\in I_f\oplus I_g$, where $I_f$ and $I_g$ are the principal ideals generated by $f$ and $g$, respectively. Then
\[
l_1\leq\sum_{i=1}^{n-k}\theta_if\quad\ \text{and}\quad l_2\leq\sum_{i=n-k+1}^{n}\theta_ig
\]
both hold for all $\theta_1,\dots,\theta_n\in(0,\infty)$ such that $\prod_{i=1}^{n}\theta_i=1$. We conclude that $l\leq 0$ and thus $\mathfrak{G}_n(f^{n-k}g^k)=0$ holds for all $k\in\{1,\dots,n-1\}$. Next let $P\colon E\to V$ be an $n$-homogeneous polynomial satisfying
\[
P(\mathfrak{G}_{n}(f_{1},\dots,f_{n}))=\check{P}(f_{1},\dots,f_{n})\quad (f_{1},\dots,f_{n}\in E^{+}).
\]
Utilizing the fact that $\mathfrak{G}_n(f^{n-k}g^k)=0$ for each $k\in\{1,\dots,n-1\}$, we have
\begin{align*}
P(f+g)&=P(f)+P(g)+\sum_{k=1}^{n-1}\binom{n}{k}\check{P}(f^{n-k}g^k)\\
&=P(f)+P(g)+\sum_{k=1}^{n-1}\binom{n}{k}P\Bigl(\mathfrak{G}_n(f^{n-k}g_k)\Bigr)\\
&=P(f)+P(g).
\end{align*}
Thus $P$ is positively orthogonally additive. The orthogonal additivity of $P$ now follows from Lemma~\ref{L:posoaeqop}. The proof is now complete.
\end{proof}

Given $n\in\mathbb{N}\setminus\{1\}$, a vector lattice $E$, and a vector space $V$, we remind the reader that an $n$-linear map $T\colon E^n\to V$ is termed \textit{orthosymmetric} if $T(f_1,\dots,f_n)=0$ whenever $f_1,\dots,f_n\in E$ and there exist $i,j\in\{1,\dots,n\}$ such that $f_i\perp f_j$. A straightforward application of the binomial theorem shows that every $n$-homogeneous polynomial $P\colon E\to V$ with $\check{P}$ orthosymmetric is orthogonally additive. Thus combining Lemma~\ref{L:posoaeqop} and Theorem~\ref{T:RMP+GM} above with the main result of \cite{Kusa}, \cite[Lemma 4]{Kusa2}, and \cite[Theorems~2.3~and~2.4]{BusSch4}, we obtain the following characterizations of bounded orthogonally additive polynomials from a uniformly complete vector lattice to a convex bornological space. Theorem~\ref{T:char} below improves \cite[Theorems~2.3~and~2.4]{BusSch4} considerably. For more information on complex vector lattices and the complexification of symmetric multilinear maps, we refer the reader to \cite[Section 1]{BusSch4}.

\begin{theorem}\label{T:char}
	Let $E$ be a uniformly complete vector lattice, let $Y$ be a convex bornological space, put $n,r\in\mathbb{N}\setminus\{1\}$, and let $P\colon E\to Y$ be a bounded $n$-homogeneous polynomial. The following are equivalent.
	\begin{itemize}
		\item[(i)] $P$ is orthogonally additive.
		\item[(ii)] $P$ is positively orthogonally additive.
		\item[(iii)] $\check{P}$ is orthosymmetric.
		\item[(iv)] $\check{P}(f^{n-k}g^{k})=0$ for every $k\in\lbrace 1,\dots,n-1\rbrace$ whenever $f\perp g$.
		\item[(v)] $P\bigl(\mathfrak{S}_{n}(f_1,\dots,f_r)\bigr)=\sum_{k=1}^{r}P(f_k)$ holds for all $f_{1},\dots,f_{r}\in E^{+}$.
		\item[(vi)] $P\bigl(\mathfrak{G}_n(f_1,\dots,f_n)\bigr)=\check{P}(f_1,\dots,f_n)$ holds for all $f_{1},\dots,f_{n}\in E^{+}$.
		\item[(vii)] $P(|z|)=\check{P}_{\C}(z^{\frac{n}{2}}(\bar{z})^{\frac{n}{2}})$ holds for all $z\in E_{\C}$ if $n$ is even, while if $n$ is odd, then $P(|z|)=\check{P}_{\C}(z^{\frac{n-1}{2}}(\bar{z})^{\frac{n-1}{2}}|z|)$ holds for every $z\in E_{\C}$.
	\end{itemize}
\end{theorem}

\bibliography{mybib}
\bibliographystyle{amsplain}

\end{document}